\documentclass{article}
\usepackage[utf8]{inputenc}
\usepackage{todonotes}
\usepackage{amsmath}
\usepackage{amsfonts}
\usepackage{amssymb}
\usepackage{hyperref}
\usepackage{graphicx,multicol}
\usepackage{enumitem}
\usepackage{authblk}
\usepackage{soul}
\usepackage{bm}
\graphicspath{{images/}}
\usepackage{tikz-cd}
\usepackage[all]{xy}
\usepackage{xcolor}
\usepackage{amsthm}
\usepackage{verbatim}

\newtheorem{thm}{Theorem}[section]

\newtheorem{prop}[thm]{Proposition}
\newtheorem{cor}[thm]{Corollary}
\newtheorem{lem}[thm]{Lemma}
\newtheorem{notn}[thm]{Notation}
\newtheorem{conj}[thm]{Conjecture}

\theoremstyle{definition}
\newtheorem{Def}[thm]{Definition}
\newtheorem{ex}[thm]{Example}

\newtheoremstyle{cited}%
  {}
  {}
  {\itshape}
  {}
  {\bfseries}
  {.}
  {.5em}
  {\thmname{#1} \thmnumber{#2} \thmnote{\normalfont#3}}
\theoremstyle{cited}

\newcommand{\Nof}{\overline{N}_{c_1}^X}
\newcommand{\nof}{N_{c_1}^X}
\newcommand{\nZof}{N_{c_1}^{\mathbb{Z}^n}}
\newcommand{\NZof}{\overline N_{c_1}^{\mathbb{Z}^n}}

\newcommand{\ZZ}{\mathbb{Z}}

\newcommand{\0}{\mathit{0}}

\DeclareMathOperator{\id}{id}

\DeclareMathOperator{\Img}{Im}
\DeclareMathOperator{\Aut}{Aut}
\DeclareMathOperator{\del}{del}
\DeclareMathOperator{\spanop}{span}

\newcommand{\Z}{\mathbb{Z}}
\newcommand{\pd}{\partial}

\renewcommand{\bar}{\overline}

\begin{document}
\title{Computability of digital cubical singular homology of $c_1$-digital images}

\author[1,*]{Samira Sahar Jamil}
\author[2]{P. Christopher Staecker}
\author[3]{Danish Ali}
\affil[1]{Department of Mathematics, University of Notre Dame, Notre Dame, Indiana}
\affil[2]{Department of Mathematics, Fairfield University, Fairfield,
Connecticut}
\affil[3]{School of Mathematics and Computer Science, Institute of Business Administration (IBA), Karachi, Pakistan}

\affil[*]{Corresponding author: Samira S. Jamil, sjamil@nd.edu}

\maketitle

\begin{abstract}
Discrete cubical homology arose as the homology theory associated with discrete cubical homotopy theory. Despite the combinatorial nature of this homology, its computation has posed a significant challenge to the researchers in the field. This paper focuses on determining the discrete cubical homology of $c_1$-digital images, which are subgraphs of the integer lattice. We compare the discrete cubical homology of $c_1$-digital images with the computationally simpler $c_1$-cubical homology as a possible route to simplifying these computations. This comparison is motivated by the classical equivalence between simplicial and singular homology theories, but the construction and proof of the chain map was found to be unexpectedly difficult. Furthermore, via the chain map constructed in this work, the $c_1$-homology, developed by the second author, is shown to be functorial and homotopy-type invariant.

\end{abstract}

\section{Introduction}

Discrete cubical homology, introduced in \cite{Barcelo_metric}, provides a combinatorial analogue of classical cubical singular homology for graphs. 
This homology theory is associated with the discrete homotopy theory of graphs \cite{Barcelo_disc_homotopy}, \cite{BBdLL06}, \cite{BGJW19}, \cite{CK24}, \cite{CK}. 
This paper presents a useful step towards efficient computations of discrete cubical homology theory in the case where the graph is a $c_1$-digital image. 

Discrete cubical homology is defined similarly to the usual cubical singular homology theory of topological spaces (described in detail in Massey’s classical book \cite{Massey91}). The basic objects of discrete cubical homology are discrete singular cubes, which are graph maps from the discrete cube to the graph. Despite the combinatorial nature of discrete cubical homology, its computation has posed a significant challenge to researchers in the field. 
For example, it was shown in \cite{Barcelo_metric} that the discrete analogue of excision only holds for dimensions 0 and 1. The main computational challenge is due to the size of the chain groups, which grows exponentially with dimension and number of vertices in the graph. 

Despite these challenges certain progress has been made. In \cite{Barcelo_connections}, Barcelo et al. showed that certain noninjective singular cubes called connections can be removed from the chain groups without disturbing the homology. Barcelo et al. \cite{Barcelo_SIAM} show that the discrete cubical homology of graphs that have no 3-cycles or 4-cycles vanishes for dimension 2 and above.

In this paper, we focus on $c_1$-digital images, which are graphs with vertex sets in $\mathbb{Z}^n$ and no diagonal adjacencies. 
Digital images are fundamental objects of study in ``digital topology", a combinatorial theory designed to study topological properties of discrete, rather than continuous, spaces.
In digital topology, digital images are modeled as subsets of $\mathbb{Z}^n$ equipped with an adjacency relation. Notions of continuity, connectedness, homotopy, and homology have been developed. Digital cubical singular homology $dH_q(X)$ was defined in \cite{Jamil_DigCubSingHom}, and results analogous to classical algebraic topology, including functoriality, homotopy invariance, and a digital Hurewicz theorem, were established. When $X$ is viewed as a reflexive graph with edges defined by adjacency, $dH_q(X)$ coincides with the discrete cubical homology $\mathcal{H}_q(X)$ described above.

However, in the context of $c_1$-digital images, another cubical homology theory, called $c_1$-cubical homology, denoted  $H_q^{c_1}(X)$, was developed by the second author  \cite{Chris_Homotopy_relations}, where chains are generated by elementary cubes in digital image $X$. The resulting chain complex is much smaller and $H_q^{c_1}(X)$ is computationally much simpler than digital cubical singular homology $dH_q(X)$.  However, $c_1$-cubical homology has the disadvantage of not being obviously functorial.   In \cite{Chris_Homotopy_relations},
a partial proof of functoriality of $H_q^{c_1}(X)$ is given based on computer enumerations. It was conjectured \cite[Conjecture 6.3]{Chris_Homotopy_relations} that $H_q^{c_1}(X)$ and $dH_q(X)$ are isomorphic.

In this paper, motivated by the classical equivalence between simplicial and singular homology, we construct a surjective chain map
between the chain complexes which define $dH_q(X)$ and $H_q^{c_1}(X)$

\[
\beta : dC_q(X) \longrightarrow C_q^{c_1}(X),
\]
constituting a first step  toward the conjectured isomorphism between $dH_q(X)$ and $H_q^{c_1}(X)$. Although injectivity of $\beta$ on homology remains open, this map establishes a concrete relationship between the two chain complexes and shows that the $c_1$-chain complex embeds naturally as a subcomplex of the discrete cubical chain complex. 

As an application, we give a complete analytical proof of functoriality for $H_q^{c_1}(X)$, extending the  partial proof of functoriality of $H_q^{c_1}(X)$ done in \cite{Chris_Homotopy_relations}.
 While the full isomorphism between $dH_q(X)$ and $H_q^{c_1}(X)$ remains open, the map $\beta$ is a first step toward an equivalence for $c_1$-digital images,
provided that one can prove that the homology of noninjective cycles is trivial, at least for certain graphs.

The conjectured isomorphism between discrete cubical homology and $c_1$-homology is further supported by known low-dimensional results. In particular, combining the main result of \cite{Barcelo_SIAM} with the digital Hurewicz theorem proved in \cite{Barcelo_metric}, one obtains that discrete cubical homology and $c_1$-homology agree for all $1$-dimensional $c_1$-digital images and that these two homology theories coincide in dimension $1$ for arbitrary $c_1$-digital images.

The rest of the paper is organized as follows. In Section~\ref{sec_prelim}, we review basic notions of digital topology and the homologies $dH_q(X)$ and $H_q^{c_1}(X)$. Sections~\ref{sec_Genprop}--\ref{sec_symmetries} classify singular cubes in $c_1$-digital images. 
Section~\ref{sec_Genprop} develops geometric properties of singular cubes in $c_1$-digital images. Section~\ref{sec_cubeauto} studies injective digital cubes and their relation to cube automorphisms, while Section~\ref{sec_symmetries} analyzes certain classes of noninjective cubes.
Using this classification of singular cubes, Section~\ref{sec_chainmap} constructs the chain map $\beta$ between the chain complexes of the two homologies. Finally, Section~\ref{sec_final} discusses examples and limitations of the conjectured equivalence for more general graphs. While it still remains open whether the two homologies are isomorphic for $c_1$-digital images, we discuss in Section \ref{sec_final} with an example that for certain more general graphs which are not $c_1$-digital images, the analogous conjecture fails.

\section{Preliminaries, singular and elementary cubical homologies}\label{sec_prelim}
A digital image $(X,\kappa)$ is a subset $X$ of 
$\mathbb{Z}^d$ with an adjacency relation, $\kappa$. 
Various adjacency relations are used in digital 
topology for digital images, to give a concept of 
proximity or closeness among elements (or pixels)
of digital images.
Concepts like continuity, connectivity, homotopy, and 
homology are developed using these adjacency relations.

In this paper, we exclusively focus on \emph{$c_1$-adjacency}, in which two points of $ \mathbb Z^n$ are regarded as adjacent if and only if their coordinates differ by 1 in one position, and are equal in all other positions. In $\mathbb Z^2$, this is referred to as ``4-adjacency'', since each point of $\mathbb Z^2$ is $c_1$-adjacent to exactly 4 other points. In $\mathbb Z^3$, this is sometimes called ``6-adjacency''. 

Since we exclusively use $c_1$-adjacency in this paper, we will rarely need to specify the particular adjacency relation, so we will write a digital image $(X,\kappa)$ or $(X,c_1)$ simply as $X$.

For digital images  $X$ and $Y$, the function 
$f:X\to Y$ is \textit{continuous} if every 
pair of adjacent points in 
$X$ maps to equal or 
adjacent points in $Y$. 
Viewing $X$ and $Y$ as reflexive graphs, continuity of $f$ is equivalent to $f$ being a graph homomorphism.

\subsection{Digital cubical singular homology, 
{$dH_q(X)$}}
\label{ssec_Digcubsing}
We give below the definition and some 
results for digital cubical singular homology 
developed in \cite{Jamil_DigCubSingHom}. 
Consider a digital image $X$ and 
the digital interval $I=[0,1]_Z$. 
For a nonnegative integer $q$, a \emph{singular digital $q$-cube} in $X$ (or simply \emph{singular $q$-cube}) is a continuous map $\sigma:I^q \to X$. 
Define $dQ_q(X)$ to be the
free abelian group generated by the set of all $q$-cubes in $X$.
We say a singular $q$-cube $\sigma$ 
in $dQ_q(X)$ is \emph{degenerate} if there is some 
coordinate $t_i$ with $1\leq i\leq q$ such that 
$\sigma(t_1,\ldots,t_q)$ does not depend on $t_i$. Otherwise we say $\sigma$ is \emph{nondegenerate}.

Let $dD_q(X)$ be the free abelian group
generated by  the 
set of all degenerate $q$-cubes. Define the quotient group
$dC_q (X)=dQ_q (X)/dD_q (X) $. 
It can be shown that $dC_q(X)$ is isomorphic to the free 
abelian group generated by the set of all 
nondegenerate $q$-cubes. 
Define ``face operators'' $A_i,B_i:I^{q-1}\to I^q$
as
$A_i(t_1,t_2,\ldots,t_{q-1})
= (t_1,\ldots,t_{i-1},0,t_i,\ldots,t_
{q-1})
$ and $B_i(t_1,t_2,\ldots,t_{q-1})
= (t_1,\ldots,t_{i-1},1,t_i,\ldots,t_
{q-1})
$, and a  homomorphism 
$\partial_q:dC_q(X)\rightarrow dC_{q-1}(X)$ 
as 
\[ \partial_q \sigma = \sum_{i=1}^q(-1)^i(A_i\sigma-B_i\sigma), \]
where: $A_i\sigma=\sigma\circ A_i$
and $B_i\sigma=\sigma\circ B_i$.
It can be shown that 
$\partial_{q-1}\circ\partial_q=0$. 
Therefore, 
$(dC_q (X),\partial_q)$ 
is a chain 
complex, and its homology is denoted $dH_q(X)$.

Elements of $dC_q(X)$, $\ker \partial_q$,
$\Img \partial_{q+1}$ and $dH_q(X)$ are called
$q$-chains, $q$-cycles, $q$-boundaries
and homology classes, respectively.
The singular $(q-1)$-cubes $A_i\sigma$ and 
$B_i\sigma$ are called front and back 
$i$-faces of $\sigma$, respectively.


%
%

Corollary 3.8 and Theorem 5.10 of \cite{Jamil_DigCubSingHom} show that $dH_q$ is homotopy-type invariant in the sense we will describe in Section \ref{sec_chainmap}. A similar result for discrete homology theory for metric spaces also appears in \cite{Barcelo_metric}.
%



\subsection{{$c_1$}-cubical homology, 
{$H^{c_1}_q(X)$}}
\label{ssec_c1cubhom}
A homology for digital images with 
$c_1$-adacency, called $c_1$-cubical homology, 
was  given in \cite{Chris_Homotopy_relations}. 
Basic definitions from 
\cite{Chris_Homotopy_relations} are given 
below:
For any integer $a$, a set of the  form 
$[a,a+1]_\mathbb{Z}=\{a,a+1\}$ or 
$[a,a]_\mathbb{Z}=\{a\}$ is called an 
\textit{elementary interval}. 
The elementary interval $[a,a]_\mathbb{Z}$ is 
called \textit{degenerate}, while the 
elementary interval $[a,a+1]_\mathbb{Z}$ is 
called \textit{nondegenerate}.
An \textit{elementary cube} $Q$ is a cartesian 
product of $n$ elementary intervals $J_i$, for 
$i\in\{1,2,\ldots,n\}$, 
\textit{i.e.}, $Q=J_1\times\cdots\times J_n$.
The \textit{dimension} of $Q$ is the number of nondegenerate elementary 
intervals $J_i$. 
If the dimension of elementary cube is $q$, we 
call it an \emph{elementary $q$-cube}.
For a $c_1$-digital image $X$ and $q\geq 
0$, let $C_q^{c_1}(X)$ be the free abelian 
group  
generated by the set of all elementary 
$q$-cubes in $X$. The face maps $A_i^{c_1},B_i^{c_1}:C_q^{c_1}(X)\to 
C_{q-1}^{c_1}(X)$ are defined as follows:
\[A_i^{c_1}Q=J_1\times \cdots\times J_{i-1}\times 
\{\min {J_i}\}
\times J_{i+1}\times\cdots\times J_n\]
\[B_i^{c_1}Q=J_1\times \cdots\times J_{i-1}\times 
\{\max{J_i}\}
\times J_{i+1}\times\cdots\times J_n\]

Note that $A_i^{c_1}Q$ and $B_i^{c_1}Q$ are 
distinct only 
if $J_i$ is nondegenerate. 
Let $\{k_1,\ldots,k_q\}\subset\{1,\ldots,n\}$ be 
the set of all indices such that $J_{k_i}$ is
nondegenerate.
The boundary operator is given as: 
$\partial_q^{c_1}Q=
\sum_{i=1}^q(-1)^i(A_{k_i}^{c_1}Q-
B_{k_i}^{c_1}Q)$.
It can be shown that 
$(C_q^{c_1}
(X),
\partial_q^{c_1})$ is a chain complex, and the associated 
homology $H_q^{c_1}(X)$ is called the
\emph{$c_1$-cubical homology} of the digital image 
$X$.

The main goals of the paper are to define a surjective chain map $\beta:dC_q(X) \to C_q^{c_1}(X)$, and to prove that the homology theory $H_q^{c_1}$ is functorial and homotopy-type invariant. Whether this chain map $\beta$ induces an isomorphism in homology is still an open question.

%



\section{General properties of singular cubes}\label{sec_Genprop}
In this section we give some basic facts about singular cubes in digital images with $c_1$-adjacency.

\subsection{$c_1$-neighborhoods}
\label{ssec_c1nhd}
For a digital image $(X,c_1)$, and $x\in X$, 
define the $c_1$-neighborhood and the deleted 
$c_1$-neighborhood of $x$, respectively as:
\begin{align*}
\Nof(x)&=\{y\in X\mid x 
\text{ is }c_1\text{-adjacent or equal to } y\}, \\
\nof(x)&=\{y\in X\mid x 
\text{ is }c_1\text{-adjacent   to } y\}.
\end{align*}

Let $\Nof(x_1,\ldots,x_n)$ denote the set 
$\Nof(x_1)\cap\cdots\cap\Nof(x_n)$ of 
common $c_1$-neighbors of elements 
$x_1,\ldots,x_n$ of $(X,c_1)$, and similarly for 
$\nof$.

\begin{prop}\label{prop_nhd_intersect}
For distinct elements $x,y$ in 
$ \mathbb{Z}^n$, the set $\nZof(x,y)$ is nonempty
if and only if one of the following is true:
\begin{itemize}
\item $x$ and $y$ differ in exactly $2$ 
coordinates by $1$ unit each.
\item $x$ and $y$ differ in exactly $1$ 
coordinate by $2$ units.
\end{itemize}
\end{prop}
\begin{proof}
Let $z\in\nZof(x,y)$. There are 
indices $j,k\in\{1,2,\ldots,n\}$ such that
\[
z_i=x_i, \forall i\neq j \text{ and } x_j\in
\{z_j\pm 1\},\]
\[
z_i=y_i, \forall i\neq k \text{ and } y_k\in
\{z_k\pm 1\}.\]
Therefore $y_i=x_i, \forall i\notin\{j,k\}$. 
Furthermore, if $j\neq k$ then $x$ and $y$ 
differ by 1 unit in coordinates $j$ and $k$, 
while if $j= k$ then $x$ and $y$ differ by 2 
units in coordinate $j$.
\end{proof}
We can conclude from Proposition 
\ref{prop_nhd_intersect} that whenever 
$c_1$-neighborhoods of two distinct elements 
intersect, they intersect in at most 
two elements. Also, the deleted $c_1$-neighborhoods 
of two $c_1$-adjacent elements do 
not intersect.  \newpage
\begin{cor}\label{cor_comn_nbrs}
The following is true for distinct elements 
$x,y$ in $\mathbb{Z}^n$:
\begin{enumerate}[label={\emph{({\roman*})}}]
\item $\left\vert\NZof(x,y)\right\vert\leq 2$.
\item If $x$ and $y$ are $c_1$-neighbors, then 
$\NZof(x,y)=\{x,y\}$ and $\nZof(x,y)=\emptyset$.
\item The elements $x$ and $y$ differ in exactly $2$ 
coordinates by $1$ unit each iff they have exactly two common neighbors, i.e.,
$\nZof(x,y)=\{z_1,z_2\}$.
\end{enumerate}
\end{cor}
The following corollary asserts that  
$c_1$-neighborhoods of more than 
$2$ points intersect in 
at most $1$ point.
\begin{cor}\label{cor_grt_thn_3}
Let $k\ge 3$ and $i\in \{1,2,\ldots,k\}$.
For distinct elements $x_i$ in $\mathbb{Z}^n$, 
$\left\vert\nZof(x_1,\ldots,x_k)
\right\vert\leq 1$. 
If $k>2n$, then $\nZof(x_1,\ldots,x_k)=\emptyset$.
\end{cor}


\begin{proof}
We have $\nZof(x_1,\dots,x_k) \subseteq \nZof(x_1,x_2,x_3)$, so it will suffice to show that $|\nZof(x_1,x_2,x_3)| \le 1$. To show this, assume that $\nZof(x_1,x_2,x_3)$ is nonempty, and we will show that it has only 1 element.
If $\nZof(x,y,z)$ is 
nonempty, then let $w\in\nZof(x,y,z)$, then $w$ differs from each of $x,y,z$ in one coordinate only, and in that coordinate by a difference of 1. The coordinate in which $w$ differs from $x,y,$ and $z$ cannot be the same for all 3 points, since the difference can only be $+1$ or $-1$. Thus we may assume without loss of generality that $w$ differs from $x$ in some coordinate and differs from $y$ in some different coordinate.

Without loss of generality we may further assume that $w$ differs from $x$ in coordinate 1 (and is equal in all other coordinates), and from $y$ in coordinate 2 (and is equal in all other coordinates). Then we must have $w=(y_1,x_2,x_3,\dots,x_n)$, with $x_i=y_i$ for $i\ge 3$. Since $w$ is uniquely determined by $x,y,z$, there can be no other element of $\nZof(x,y,z)$. Thus we have shown $\left\vert\nZof(x_1,\ldots,x_k) \right\vert\leq 1$. 

For the second statement, note that the common neighborhood of more than $2n$ elements is 
empty, because an element in $\ZZ^n$ has at most $2n$ 
$c_1$-neighbors.
\end{proof}

\subsection{The elementary cube 
{$I^q$} and the digital {$q$}-cube}
\label{ssec_eltcub_digcub}

Partition the set $I^q$ into sets 
\[ S_n=\left\{(t_1,t_2,\ldots,t_q) \mid \sum_{j=1}^qt_j=n\right\}, n
\in\{0,1,\ldots,q\}. \]
Clearly the $S_n$ form a partition of the cube $I^q$, and $S_0$ and $S_n$ are singletons.
We develop below some notation for elements of 
$S_n$.
\begin{notn}[$\theta$-notation for elements of 
$I^q$]~\\
For elementary $q$-cube $I^q$, let $\0$ denote
the only element of $S_0$.
For a set $\{i_1,\ldots,i_n\}$ of $n$ distinct
indices in $\{1,\ldots,q\}$,
let $\theta^{i_1i_2\cdots i_n}$ 
denote the element of $S_n$, 
with $1$s in coordinates 
$i_1,\ldots,i_n$, and $0$s everywhere else.
If $n=0$, then $\theta^{i_1i_2\cdots i_n}=\0$.
\end{notn}
The elements $\theta^{i_1i_2\cdots i_n}$ and 
$\theta^{l_1l_2\cdots l_n}$ are equal if and only if the sets 
$\{i_1,i_2,\ldots, i_n\}$  and $\{l_1,l_2,\ldots, 
l_n\}$ are equal. Therefore the order of superscripts 
does not make any difference.
For $1\leq n\leq q-1$, each element in 
$S_n$ has $c_1$-neighbors only in $S_{n-1}$ and 
$S_{n+1}$. 

The next two results show how two different elements of $\sigma(I^q)$ differ from each other when $\sigma$ is injective on a certain subset of $I^q$.

\begin{lem}\label{lem_injec_diff}
    Let $\sigma:I^q\to X$ be a singular $q$-cube. If the restriction $\sigma\vert_{S_1\cup\cdots\cup S_m}$ is injective for $m\geq 
 2$, then two elements $\sigma(\theta^{i_1\cdots i_n})$ and $\sigma(\theta^{l_1\cdots l_k})$ for $0\leq n,k\leq m$, differ from each other in exactly as many coordinates as there are elements in the symmetric difference of sets $\{i_1,\ldots,i_n\}$ and $\{l_1,\ldots,l_k\}$.
\end{lem}
\begin{proof}
We will compare the coordinates of $\sigma(\theta^{i_1\dots i_n})$ and $\sigma(\theta^{l_1 \dots l_k})$ by comparing each to $\sigma(\0)$. 
Let $c_i$ be the single coordinate in which $\sigma(\theta^i)$ differs from $\sigma(\0)$. It will suffice to show that $\sigma(\theta^{i_1\dots i_n})$ differs from $\sigma(\0)$ in exactly the coordinates $c_{i_1},\dots, c_{i_n}$, for then $\sigma(\theta^{i_1\dots i_n})$ and  $\sigma(\theta^{l_1 \dots l_k})$ would differ from each other in coordinates $c_{i_1},\ldots,c_{i_n}$ and $c_{l_1},\ldots,c_{l_k}$, and equal to each other in all the other coordinates.

First we argue that $c_i\neq c_j$ when $i\neq j$. Observe that $\sigma(\0)$ and $\sigma(\theta^{ij})$ have exactly 2 distinct common neighbors, namely $\sigma(\theta^{i})$ and $\sigma(\theta^{j})$. Thus by Corollary \ref{cor_comn_nbrs} (iii), the points $\sigma(\0)$ and $\sigma(\theta^{ij})$ differ in exactly two distinct coordinates $c_i$ and $c_j$.

Now we show that $\sigma(\0)$ differs from $\sigma(\theta^{i_1\cdots i_n})$ in coordinates $c_{i_1},c_{i_2},\ldots,c_{i_n}$ by induction on $n$. The argument above forms the base case of the induction when $n=1$ or $2$. From Corollary \ref{cor_comn_nbrs} (iii), $\sigma(\theta^{i_1\cdots i_{n-2}})$  and  $\sigma(\theta^{i_1\cdots i_n})$ differ from each other in coordinates $c_{i_{n-1}}$ and $c_{i_{n}}$ and agree in all the other coordinates. By induction, $\sigma(\0)$ and $\sigma(\theta^{i_1\cdots i_{n-2}})$ differ exactly in coordinates $c_{i_1},\dots,c_{i_{n-2}}$. Thus we conclude that $\sigma(\0)$  and  $\sigma(\theta^{i_1\cdots i_n})$ differ in coordinates $c_{i_1},\ldots,c_{i_{n}}$.
\end{proof}

We prove below that if the digital 
$q$-cube $\sigma$ is injective when restricted 
to a certain subset of $I^q$, 
then $\sigma$ is injective.
\begin{lem}\label{lem_sig_inj}
Let $\sigma:I^q\to X$ be a singular $q$-cube, 
$q\geq 2$ in digital image $(X,c_1)$. If the 
restriction of $\sigma$ to $S_0\cup S_1\cup S_2$ 
is injective, then $\sigma$ is injective. 
\end{lem}

\begin{proof}
Let $q>2$ because the statement is trivially true 
for $q=2$.
It suffices to show that if $\sigma\vert_{S_0\cup 
\cdots\cup S_{n-1}}$ is injective, then $\sigma
\vert_{S_0\cup \cdots\cup S_{n}}$ is injective, 
for $3\leq n\leq q $.
 We prove the injectivity of $\sigma
\vert_{S_0\cup \cdots\cup S_{n}}$  by showing that $\sigma(\theta^{i_1\cdots i_n})\neq \sigma(\theta^{l_1\cdots l_m})$, for $0\leq m\leq n$, when the sets $\{i_1,\ldots,i_n\}$ and $\{l_1,\ldots,l_m\}$ are not equal.
We consider 3 different cases for the sets $\{i_1,\ldots,i_n\}$ and $\{l_1,\ldots,l_m\}$ below.

\textit{Case 1}: $0\leq m<n$ and $\{i_1,\ldots,i_n\}\cap \{l_1,\ldots,l_m\}\neq\emptyset$~~
Let $j$ be a common element  in both $\{i_1,\ldots,i_n\}$ and $\{l_1,\ldots,l_m\}$. Since $\{i_1,\ldots,i_n\}$ has more elements than $\{l_1,\ldots,l_m\}$, 
there is $i_\alpha\in \{i_1,\ldots,i_n\}$ such that $i_\alpha\notin
\{l_1,\ldots,l_m\}$. Assuming that $\sigma(\theta^{i_1\cdots i_n})= \sigma(\theta^{l_1\cdots l_m})$,
 the element $\sigma(\theta^{i_1\cdots i_n \widehat{j}})$ is a $c_1$-neighbor of $\sigma(\theta^{l_1\cdots l_m})$, which is not possible because the symmetric difference of the sets $\{i_1,\ldots,i_n\}-\{j\}$ and $\{l_1,\ldots,l_m\}$ contains at least two elements $j$ and $i_\alpha$, and by Lemma \ref{lem_injec_diff}, $\sigma(\theta^{i_1\cdots i_n \widehat{j}})$ and $\sigma(\theta^{l_1\cdots l_m})$ differ in at least two coordinates.

\textit{Case 2}: $0\leq m<n$ and $\{i_1,\ldots,i_n\}\cap \{l_1,\ldots,l_m\}=\emptyset$~~
 Assuming that $\sigma(\theta^{i_1\cdots i_n})= \sigma(\theta^{l_1\cdots l_m})$,
 the element $\sigma(\theta^{i_1\cdots i_{n-1}} )$ is a $c_1$-neighbor of $\sigma(\theta^{l_1\cdots l_m})$, which is not possible because the sets $\{i_1,\ldots,i_{n-1}\}$ and $\{l_1,\ldots,l_m\}$ have at least $m+n-1\geq 2$ elements in their symmetric difference, and by Lemma \ref{lem_injec_diff} $\sigma(\theta^{i_1\cdots i_{n-1}} )$  and  $\sigma(\theta^{l_1\cdots l_m})$ differ in at least two coordinates.
 
 \textit{Case 3}: $m=n$~~
 Since both the sets $\{i_1,\ldots,i_n\}$ and $ \{l_1,\ldots,l_n\}$ are different from each other but have same number of elements, 
there is $i_\alpha\in \{i_1,\ldots,i_n\}$ such that $i_\alpha\notin
\{l_1,\ldots,l_n\}$ and also there is $l_\beta\in
\{l_1,\ldots,l_n\}$ such that $l_\beta\notin\{i_1,\ldots,i_n\}$. Consider some $i_j$ different from  $i_\alpha$, and some $l_k$ different from $l_\beta$.
If $\sigma(\theta^{i_1\cdots i_n})= \sigma(\theta^{l_1\cdots l_n})$,
then the elements $\sigma(\theta^{i_1\cdots \widehat{i_j}\cdots i_n })$ and $\sigma(\theta^{l_1\cdots \widehat{i_k}\cdots l_n })$ are $c_1$-neighbors of $\sigma(\theta^{i_1\cdots i_n})= \sigma(\theta^{l_1\cdots l_n})$. The symmetric difference of $\{i_1,\ldots,i_n\}-\{i_j\}$ and $\{l_1,\ldots,l_n\}-\{l_k\}$ consists of at least 3 elements: $i_\alpha$, $l_\beta$, $i_j$ and $l_k$ (where $i_j$ and $l_k$ may possibly be equal), and by Lemma \ref{lem_injec_diff} 
$\sigma(\theta^{i_1\cdots \widehat{i_j}\cdots i_n })$ and $\sigma(\theta^{l_1\cdots \widehat{i_k}\cdots l_n })$ differ from each other in at least 3 coordinates. But then by Proposition \ref{prop_nhd_intersect}, $\sigma(\theta^{i_1\cdots \widehat{i_j}\cdots i_n })$ and $\sigma(\theta^{l_1\cdots \widehat{i_k}\cdots l_n })$ can not share a common $c_1$-neighbor.
\end{proof}

%
Now we show that for an injective cube $\sigma$, the restriction to $S_0 \cup S_1$ completely determines all other values of $\sigma$.
\begin{lem}\label{cube_unique}
Let $\sigma$ be a singular $q$-cube with $q\ge 2$. If $\sigma$ is injective, then all values of $\sigma$ are uniquely determined by its values on $S_0\cup S_1$.
\end{lem}
\begin{proof}

Let $\sigma$ be injective. We prove the Lemma by showing that, if the values of $\sigma$ on $S_0\cup \dots \cup S_n$ are given for some $n\ge 1$, then these values will uniquely determine any value of $\sigma$ on $S_{n+1}$. 

For $n=1$, we must show that the values of $\sigma$ on $S_2$ are uniquely determined by the values of $\sigma$ on $S_0 \cup S_1$. Take some $x = \theta^{i_1i_2} \in S_2$. Then $\sigma(x)$ is a common neighbor of $\sigma(\theta^{i_1})$ and $\sigma(\theta^{i_2})$, which are distinct because $\sigma$ is injective. The point $\sigma(\0)$ is also a common neighbor of $\sigma(\theta^{i_1})$ and $\sigma(\theta^{i_2})$. by Corollary \ref{cor_comn_nbrs} there are at most two common neighbors of $\sigma(\theta^{i_1})$ and $\sigma(\theta^{i_2})$. And because $\sigma$ is injective, $\sigma(x) \neq \sigma(\0)$ and thus $\sigma(x)$ is the unique common neighbor of $\sigma(\theta^{i_1})$ and $\sigma(\theta^{i_2})$ different from $\sigma(\0)$. Thus the value of $\sigma(x)$ is determined uniquely by the values of $\sigma$ on $S_0 \cup S_1$.

Now we consider the case $n\ge 2$. Assume that the values of $\sigma$ on $S_0\cup \dots \cup S_n$ are specified, and take some point $x = \theta^{i_1\dots i_n i_{n+1}} \in S_{n+1}$. Then $\sigma(x)$ is a common neighbor of every element of the form $\sigma(\theta^{i_1\dots \hat i_k \dots i_{n+1}})$. Since $\sigma$ is injective, $\sigma(x)$ is a common neighbor of $n+1 \ge 3$ points and so the value $\sigma(x)$ is uniquely determined.
\end{proof}

\section{Cube automorphisms and injective singular cubes}\label{sec_cubeauto}
Much of our work will be based fundamentally on the well-known structure of the automorphism group of the cube $I^q$, which is called the \emph{hyperoctahedral group}. See a basic description in \cite{baake}. We will summarize the relevant details.

Given any permutation $\pi$ in the symmetric group $\Sigma_q$, there is a natural automorphism $r_\pi: I^q \to I^q$ given by $r_\pi(t_1,\dots t_q) = (t_{\pi(1)},\dots, t_{\pi(q)})$. 
For simplicity we will abuse notation and write the permutation $\pi$ itself as the map $\pi = r_\pi: I^q \to I^q$. We will focus particularly on the transposition $\tau_{ij} \in \Sigma_q$, defined by $\tau_{ij}(i)=j$ and $\tau_{ij}(j)=i$ and $\tau_{ij}(k)=k$ for any $k\not\in\{i,j\}$. 

The other type of fundamental symmetry of the cube is expressed by axis reflections: given some $i\in \{1,\dots,q\}$, the reflection $T_i:I^q \to I^q$ is the map given by 
\[ T_i(t_1,\dots,t_q) = (t_1,\dots,t_{i-1},1-t_i,t_{i+1}, \dots, t_q). \] 
For convenience, we write $T_0$ as the identity map. 

More generally, let $\Z_2 = \{0,1\}$ be the group of integers modulo 2 with operation $\oplus$. Given any bit-vector $z = (z_1,\dots,z_q) \in \Z_2^q$, the \emph{translation} by $z$ is $T_z:I^q \to I^q$ given by $T_z(t_1,\dots t_q) = (t_1 \oplus z_1,\dots, t_q \oplus z_q)$. It is clear that if $z$ has nonzero coordinates in positions $k_1,\dots, k_n$, then $T_z = T_{k_1} \circ \dots \circ T_{k_n}$.

Occasionally, we will use two more combinations of the symmetries mentioned above. The composition  $T_i\circ \tau_{ij}=\tau_{ij}\circ T_j$,  is a one-quarter \emph{rotation} of the cube $I^q$ in the plane defined by $i$ and $j$.
Secondly, the ``shift'' permutation given for $i<j$ as: 
\[S_{ij}=\tau_{i,i+1}\circ\cdots\circ\tau_{j-2,j-1}\circ\tau_{j-1,j}\]
which takes the coordinate $j$ to the position of coordinate $i$.

If $\alpha:I^q \to I^q$ is a continuous bijection with continuous inverse, we call it a \emph{cube automorphism}. We write the set of all cube automorphisms on $I^q$ as $\Aut(I^q)$. The fundamental characterization of cube automorphisms in terms of permutations and translations can be described as follows:
\begin{thm}\label{hyperoctahedralthm}
Let $\alpha:I^q \to I^q$ be a cube automorphism. Then there are unique $\pi\in \Sigma_n$ and $z\in \Z_2^q$ with $\alpha(t) =z \oplus \pi(t)$.
\end{thm}



When $X$ is a digital image and $\sigma:I^q \to X$ is an injective digital cube, we can express $\sigma$ in terms of a cube automorphism $\alpha\in \Aut(I^q)$ as follows: We call a map $J: I^q \to X \subset \Z^n$ an \emph{affine embedding} when we have $J(t) = Mt + x$ for some $x\in X$, where $M$ is a $n\times q$ matrix obtained by inserting $n-q$ rows of zeroes into the $q\times q$ identity matrix. (Equivalently, $M$ is obtained by deleting $n-q$ columns from the $n\times n$ identity matrix.) In this context we call $M$ an \emph{embedding matrix}.

Theorem \ref{hyperoctahedralthm} gives a standard form for cube automorphisms, which can be used to obtain a similar form for any injective singular cube $\sigma:I^q \to X$. First we discuss some properties of injective cubes in general:

\begin{Def}
For a singular $q$-cube $\sigma:I^q \to X$, we say $\sigma$ is an \emph{embedding} if $\sigma(I^q)$ is an elementary $q$-cube in $X$. 
\end{Def}

The term ``embedding'' is used because an embedding is an isomorphism onto its image:
\begin{thm}\label{emb_is_isom}
If $\sigma:I^q \to X$ is an embedding, then $\sigma: I^q \to \sigma(I^q)$ is an isomorphism of digital images.
\end{thm}
\begin{proof}
If $\sigma:I^q \to X$ is an embedding, then $\sigma(I^q)$ is an elementary $q$-cube, and so $\sigma(I^q)$ is naturally isomorphic (as a digital image) to $I^q$. Let $i:\sigma(I^q) \to I^q$ be this isomorphism. Then $\sigma \circ i: I^q \to I^q$ is continuous and injective (and thus bijective). By Lemma 2.3 of \cite{hmps}, $\sigma \circ i$ is an isomorphism, and thus $\sigma$ is an isomorphism.
\end{proof}

Then we can easily obtain:
\begin{thm}\label{thm_inj_is_emb}
If $\sigma:I^q \to X$ is injective, then $\sigma$ is an embedding. 
\end{thm}
\begin{proof}
Let $\sigma$ be injective. Then by Lemma \ref{cube_unique}, $\sigma$ is uniquely determined by its values on the set $V = S_0 \cup S_1 = \{0, \theta^1, \dots, \theta^q\}$. There is an embedding which agrees with $\sigma$ on $V$: the affine linear map which sends $V$ to $\sigma(V)$. Since $\sigma$ is uniquely determined by its values on $V$, we conclude that $\sigma$ is this affine linear map, which is an embedding.
\end{proof}

Now we can derive our standard form for injective cubes in terms of the hyperoctahedral group.

\begin{thm}\label{singularhyperoctahedral}
Let $\sigma: I^q \to X$ be an injective cube. Then there are unique $z\in \Z^q_2$ and $\pi\in \Sigma_q$, and a unique embedding matrix $M$ such that
\[ \sigma(t) = M(z\oplus \pi(t)) + x, \]
where $x = \sigma(0)-Mz$ is the lexicographical minimum element of $\sigma(I^q)$.
\end{thm}
\begin{proof}
By Theorem \ref{thm_inj_is_emb},  the set $\sigma(I^q)$ is an elementary $q$-cube. Then there is a unique affine embedding $J(t) = Mt + x$ with $\sigma(I^q) = J(I^q)$. Specifically, $M$ is the $n\times q$ matrix with columns $e_{k_i} \in \Z^n$ where $1 \le k_1<\dots<k_q\le n$ and $\{k_1,\dots,k_q\}$ are the coordinates in which the elementary $q$-cube $\sigma(I^q)$ is nondegenerate. In this case the transpose $M^T$ acts like an inverse to $M$ in that $M^TM$ is the $q\times q$ identity matrix, and $MM^T$ is the $n\times n$ identity matrix with column $i$ replaced by 0 whenever $k\not\in \{k_1,\dots,k_q\}$.

For $t\in I^q$, let $\rho$ be defined by $\rho(t) = M^T(\sigma(t)-x)$. By the above, we have $\rho(t)\in I^q$ for all $t$, and so $\rho$ is a cube automorphism $\rho:I^q\to I^q$. By Theorem \ref{hyperoctahedralthm}, $\rho$ is expressed uniquely as $\rho(t) = z \oplus \pi(t)$ for some $z\in\Z^q_2$ and $\pi\in \Sigma_q$. Then we have:
\[ M^T(\sigma(t)-x) = z \oplus \pi(t), \]
and so
\[ MM^T(\sigma(t)-x) = M(z\oplus \pi(t)). \]
By the particular structure of $MM^T$, and since $\sigma(t)-x$ has nonzero entries only in coordinates $k_1,\dots,k_q$, we have $MM^T(\sigma(t)-x)= \sigma(t)-x$, and rearranging above gives $\sigma(t) = M(z\oplus \pi(t)) + x$ as desired.
\end{proof}


For the next lemma we will introduce a specific construction on permutations. Given $\pi\in S_n$, viewing $\pi$ as a permutation of the set $\{1,\dots,n\}$, and for some chosen $i\in \{1,\dots,n\}$, 
we will define a permutation $\pi'\in S_{n-1}$ to be the permutation obtained by deleting the $i$th entry of $\pi$, and then relabeling the remaining entries as $1$ through $n-1$, preserving their order. Homberger \cite{homberger} calls this the \emph{deletion} of entry $i$, his notation is $\pi' = \del(\pi,i)$. In other papers this operation is called ``normalization'' or ``standardization'' of the permutation.


Specifically we can describe $\pi'$ as follows:
\[
\pi'(k) = \begin{cases}
\pi(k) &$ if $ k<i $ and $ \pi(k)<\pi(i), \\
\pi(k+1) &$ if $ k\ge i$ and $\pi(k)<\pi(i), \\
\pi(k) - 1 &$ if $ k<i$ and $\pi(k)\ge \pi(i), \\
\pi(k+1)-1 &$ if $ k\ge i$ and $\pi(k)\ge \pi(i).
\end{cases}
\]

The parity of $\pi'$ is related to the parity of $\pi$ as follows:
Let $c, d \in \Sigma_n$ be the permutations with cycle notation $c = (i (i+1) \dots n)$ and $d = (n(n-1)\dots \pi(i))$. Using the formula in cases above, it is straightforward to show that $d \circ \pi \circ c(n)=n$ and $d\circ \pi \circ c(k) = \pi'(k)$ for all $k<n$. Thus the cycle structure of $\pi'$ is the same as that of $d\circ \pi \circ c$. In particular these permutations have the same parity, so we have:
\begin{equation}\label{delparity}
(-1)^{\pi'} =  (-1)^d (-1)^c (-1)^\pi = (-1)^{n-i} (-1)^{n-\pi(i)} (-1)^\pi = (-1)^{i+\pi(i)} (-1)^\pi
\end{equation}

Theorem \ref{singularhyperoctahedral} can be used to evaluate $\sigma$ specifically in coordinates. Below, we will use the following simple fact for $a,b\in \Z_2$:
\begin{equation}\label{opluseq} 
a \oplus b = a + b(-1)^a,
\end{equation}
and we write $e_i\in \Z^n$ as the $i$th standard basis vector.

\begin{lem}\label{sigmacoordinates}
Let $\sigma:I^q \to X$ be an injective cube with $\sigma(t) = M(z\oplus \pi(t)) + x$ for some $z\in \Z^q$ and $\pi\in \Sigma_q$ and $x\in X$. Let $t\in I^q$ be written in coordinates as $t = \sum_{k=1}^q t_ke_k$. Then we have:
\[ \sigma(t) = \sigma(0) + \sum_{k=1}^q (-1)^{z_k}t_{\pi^{-1}(k)} Me_k. \]
In particular, when $t=e_i$, the only nonzero term of the sum is when $k=\pi(i)$, and so we have
\[ \sigma(e_i) = \sigma(0) + (-1)^{z_{\pi(i)}} M e_{\pi(i)}. \]
\end{lem}
\begin{proof}
Recall from Theorem \ref{hyperoctahedralthm} that $x = \sigma(0)-Mz$. 

Now we evaluate:
\[ 
\begin{split}
\sigma(t) &= M(z\oplus \pi(t)) + x = M\left( \sum_{k=1}^q (z_k \oplus \pi(t)_k) e_k \right) + \sigma(0) - Mz \\
&= M\left( \sum_{k=1}^q (z_k \oplus t_{\pi^{-1}(k)}) e_k \right) + \sigma(0) - M\left( \sum_{k=1}^q z_ke_k\right) \\
&= M\left( \sum_{k=1}^q ((z_k \oplus t_{\pi^{-1}(k)}) - z_k) e_k\right) + \sigma(0) \\
&= \sigma(0) + \sum_{k=1}^q ((z_k \oplus t_{\pi^{-1}(k)}) - z_k) Me_k.
\end{split}
\]
and applying \eqref{opluseq} gives
\[ \sigma(t) = \sigma(0) + \sum_{k=1}^q (-1)^{z_k} t_{\pi^{-1}(k)} Me_k \]
as desired.
\end{proof}

The following lemma gives the precise structure of the faces of $\sigma$ in terms of the structure in Theorem \ref{singularhyperoctahedral}. 

\begin{lem}\label{faceperm}
Let $\sigma:I^q \to X$ be an injective cube with $\sigma(t) = M(z\oplus \pi(t)) + x$ for some $z\in \Z^q$ and $\pi\in \Sigma_q$ and $x\in X$. Then the faces of $\sigma$ have the form 
\[ \begin{split}
A_i\sigma(t) &= M'(z'\oplus \pi'(t)) + x', \text{and}\\
B_i\sigma(t) &= M'(z'\oplus \pi'(t)) + x' + (-1)^{z'_{\pi'(i)}} M'e_{\pi'(i)}, 
\end{split}
\] 
where $z'$ is obtained by deleting entry $i$ from $z$, the permutation $\pi'$ is the permutation induced from $\pi$ by omitting $\pi(i)$, the matrix $M'$ is obtained by deleting column $\pi(i)$ from $M$, and $x' = \sigma(0) - M'z'$.
\end{lem}
\begin{proof}
First we prove the statement for $A_i\sigma$. 
Let $\rho(t) = M'(z' \oplus\pi'(t)) + x'$, and we must show that $A_i\sigma = \rho$. By Lemma \ref{cube_unique}, it suffices to show that $A_i\sigma$ and $\rho$ agree on $S_0 \cup S_1$, that is, they agree when evaluated at 0 or at any standard basis vector $e_k$.

For evaluation at $0$, we have:
\[ \rho(0) = M'(z' \oplus \pi'(0)) + x' = M'z' + \sigma(0) - M'z' = \sigma(0) = A_i\sigma(0). \]

Now we show that $\rho(e_k) = A_i\sigma(e_k)$ for some standard basis vector $e_k$. We will assume that $k<i$, otherwise several $k$ should be replaced with $k+1$ below. By Lemma \ref{sigmacoordinates} we have:
\[  A_i \sigma(e_k) = \sigma(e_k) = \sigma(0) + (-1)^{z_{\pi(k)}} Me_{\pi(k)}, \]
and applying Lemma \ref{sigmacoordinates} to $\rho(t) = M'(z' \oplus\pi'(t)) + x'$, similarly gives:
\[ \rho(e_k) = \sigma(0) + (-1)^{z'_{\pi'(k)}} M'e_{\pi'(k)}. \]
Thus in order to show that $\rho(e_k) = A_i\sigma(e_k)$, we must show that $(-1)^{z_{\pi(k)}} Me_{\pi(k)} = (-1)^{z'_{\pi'(k)}} M'e_{\pi'(k)}$.

Since $z'$ arises from $z$ by deleting entry $\pi(k)$, and $\pi'$ is the permutation induced from $\pi$ by deleting $\pi(k)$, we will have $z_{\pi(k)} = z'_{\pi'(k)}$. Similarly since $M'$ arises from $M$ be deleting column $\pi(k)$, we have $Me_{\pi(k)} = M'e_{\pi'(k)}$ (these are each equal to column $\pi(k)$ of $M$). Thus we do indeed have $(-1)^{z_{\pi(k)}} Me_{\pi(k)} = (-1)^{z'_{\pi'(k)}} M'e_{\pi'(k)}$, and the result follows.

For $B_i\sigma$, we repeat the same arguments. Let: 
\[ \omega(t) = M'(z'\oplus \pi'(t)) + x' + (-1)^{z'_{\pi'(i)}}M' e_{\pi'(i)}, \]
and we will show $B_i\sigma = \omega$ by showing that they agree at $0$ and $e_k$.

For evaluation at $0$, we have $B_i\sigma(0) = \sigma(e_i)$, and so using Lemma \ref{sigmacoordinates} we have:
\[ B_i\sigma(0) = \sigma(e_i) = \sigma(0) + (-1)^{z_{\pi(i)}} Me_{\pi(i)}. \]
We also have:
\[ 
\begin{split}
\omega(0) &= M'z' + x' + (-1)^{z'_{\pi'(i)}}M' e_{\pi'(i)} \\
&= A_i\sigma (0) + (-1)^{z'_{\pi'(i)}}M' e_{\pi'(i)} 
= \sigma(0) + (-1)^{z'_{\pi'(i)}}M' e_{\pi'(i)} 
\end{split}
\]
and so $B_i\sigma(0) = \omega(0)$ as desired.

For evaluation at $e_k$, again we assume that $k<i$. We have: $B_i\sigma(e_k) = \sigma(e_i+e_k)$, and this can be computed by Theorem \ref{sigmacoordinates}. The only nonzero terms of the sum of Theorem \ref{sigmacoordinates} are terms $\pi^{-1}(i)$ and $\pi^{-1}(k)$, and we have:
\[ B_i\sigma(e_k) = \sigma(0) + (-1)^{z_{\pi(i)}}Me_{\pi(i)} + (-1)^{z_{\pi(k)}} Me_{\pi(k)}. \]
We also have:
\[ \omega(e_k) = M'(z'\oplus \pi'(e_k)) + x' + (-1)^{z'_{\pi'(i)}}M' e_{\pi'(i)} = A_i\sigma(e_k) + (-1)^{z'_{\pi'(i)}}M' e_{\pi'(i)}.\]
Applying Theorem \ref{sigmacoordinates} to compute $A_i\sigma(e_k)$ above gives:
\[ 
\begin{split}
\omega(e_k) &= A_i \sigma(0) + (-1)^{z'_{\pi'(k)}} M'e_{\pi'(k)} + (-1)^{z'_{\pi'(i)}}M' e_{\pi'(i)} \\
&= \sigma(0) + (-1)^{z_{\pi(k)}} Me_{\pi(k)} + (-1)^{z_{\pi(i)}}M e_{\pi(i)} = B_i\sigma(e_k) 
\end{split}
\]
as desired.
\end{proof}

The face operators behave predictably with respect to the maps $T_i$ and $\tau_{ij}$ as follows. The listed identities can all be easily shown by direct computation.
\begin{prop}\label{faceformulas}
Let $\sigma$ be a $q$-cube with $i\neq j$ and $k\not\in \{i,j\}$. Then:
\begin{align*}
A_k(\sigma\circ T_i) &=
\begin{cases}
A_k\sigma \circ T_{i-1},& k<i\\
B_k\sigma, & k=i\\
A_k\sigma \circ T_{i},& k>i
\end{cases}, \\
B_k(\sigma\circ T_i) &=
\begin{cases}
B_k\sigma \circ T_{i-1},& k<i\\
A_k\sigma, & k=i\\
B_k\sigma \circ T_{i},& k>i
\end{cases}, \\
A_k(\sigma\circ \tau_{ij})& =\begin{cases}
A_k\sigma\circ \tau_{i-1\,j-1},&k<i<j\\
A_j\sigma\circ S_{i\,j-1},&k=i<j\\
A_k\sigma\circ \tau_{i\,j-1},&i<k<j\\
A_i\sigma\circ S_{j-1\,i},&i<k=j\\
A_k\sigma\circ \tau_{ij},&i<j<k
\end{cases},\\
B_k(\sigma\circ \tau_{ij}) &=\begin{cases}
B_k\sigma\circ \tau_{i-1\,j-1},&k<i<j\\
B_j\sigma\circ S_{i,j-1},&k=i<j\\
B_k\sigma\circ \tau_{i\,j-1},&i<k<j\\
B_i\sigma\circ S_{j-1,i},&i<k=j\\
B_k\sigma\circ \tau_{ij},&i<j<k
\end{cases}.
\end{align*}
\end{prop}

From the above we can easily compute:
\begin{prop}\label{boundaryflipswap}
Let $\sigma$ be a $q$-cube with $i<j$. Then:
\begin{align*} 
\pd_q(\sigma \circ T_i) &= 
\sum_{k<i}(-1)^k(A_k\sigma\circ T_{i-1}-B_k
\sigma\circ T_{i-1})+(-1)^i(B_i\sigma-A_i\sigma)
\\
&
\qquad +\sum_{k>i}^{q}(-1)^k(A_k\sigma\circ T_i -B_k
\sigma\circ T_{i}). \\
\partial_q(\sigma \circ \tau_{ij}) &= \sum_{k<i}(-1)^{k}(A_k\sigma\circ 
\tau_{i-1,j-1}-B_k\sigma\circ \tau_{i-1,j-1})
\\
&\qquad+(-1)^i(
A_j\sigma\circ S_{i,j-1}-B_j\sigma\circ S_{i,j-1})
\\
&\qquad+\sum_{i<k<j}(A_k\sigma\circ \tau_{i,j-1}-
B_k\sigma\circ \tau_{i,j-1})
\\
&\qquad+(-1)^j(A_i\sigma\circ 
S_{j-1,i}-B_i\sigma\circ S_{j-1,i})\\
&\qquad+\sum_{j<k}
(-1)^k(A_k\sigma\circ \tau_{i,j}
-B_k\sigma\circ \tau_{i,j}).
\end{align*}
\end{prop}

\section{Symmetries of compatible cubes, and noninjective $q$-cubes with at least one injective face}\label{sec_symmetries}
We say that singular $q$-cubes $\sigma$ and $\gamma$ are \emph{compatible} if
$\sigma(t_1,\ldots,t_q)$ is adjacent to
$\gamma(t_1,\ldots,t_q)$, for all $t_i\in\{0,1\}$.
Geometrically speaking, this means that the two $q$-cubes are situated in such a way that they can be ``concatenated'' into a $(q+1)$-cube whose opposite faces in some coordinate are $\sigma$ and $\gamma$.

In this section we prove some useful facts about compatible injective cubes, and also consider noninjective cubes with at least one injective face.

First, it is easy to see that when two injective cubes have the same image, they differ by an automorphism.

\begin{lem}\label{swapflipcirc}
Let $\sigma, \phi: I^q \to X$ be injective singular $q$-cubes in $X$ with $\sigma(I^q) = \phi(I^q)$. Then:
\[ \phi = \sigma \circ \alpha, \]
where $\alpha \in \Aut(I^q)$.
\end{lem}
\begin{proof}
Since both $\phi$ and $\sigma$ are injective, both are embeddings by Theorem \ref{thm_inj_is_emb}, and thus are isomorphisms onto their image by Theorem \ref{emb_is_isom}. Since they have the same image, we may define $\alpha = \sigma^{-1} \circ \phi$, which will be an automorphism of $I^q$ satisfying $\phi = \sigma \circ \alpha$.
\end{proof}

When two injective cubes are compatible and have equal images, they can differ only by very specific types of automorphism.
\begin{lem}\label{lem_compatible_embeddings}
Let $\sigma, \phi$ be injective $q$-cubes which are compatible, with $\sigma(I^q)=\phi(I^q)$. Then one of the following holds:
\begin{enumerate}
\item\label{equal} $\phi = \sigma,$
\item\label{F} $\phi = \sigma \circ T_j$ for some $j$, and
\item\label{FC} $\phi = \sigma \circ T_j \circ \tau_{ij}$ for some $i\neq j$.
\end{enumerate}
\end{lem}
\begin{proof}
By Lemma \ref{swapflipcirc}, we have $\phi = \sigma \circ \alpha$, where $\alpha\in \Aut(I^q)$. By Theorem \ref{hyperoctahedralthm}, we have $\phi = \sigma \circ T_z \circ \pi$ for $\pi \in \Sigma_q$ and $T_z \in \Z_2^q$. Let $i_1,\dots,i_n$ be the nonzero coordinates of $z$, possibly with $n=0$. We must show that $n\in \{0,1\}$, and that $\pi$ is either identity or a transposition of the required form in \ref{FC}. 

First we show that $n \in \{0,1\}$. Let $x = \pi^{-1}(z)$. Then: 
\[ 
\phi(x) = \sigma \circ T_z (z) = \sigma(\0). 
\]
Since $\sigma$ and $\phi$ are compatible, $\sigma(x)$ is $c_1$-adjacent to $\phi(x)=\sigma(\0)$. Since $\sigma$ is an embedding, this means $\0$ is $c_1$-adjacent to $x$, and so either $x=\0$, or $x=\theta^j$ for some $j$, and thus $z = \pi(x)$ is a point having at most $1$ nonzero coordinate. Thus $n\le 1$ as desired.

Thus far we have shown that 
\[ \phi = \sigma \circ T \circ \pi, \]
where $T$ is either identity or $T=T_j$ for some $j$, and $\pi$ is a permutation.

Next we will show that $\pi(k)=k$ for all except at most 3 values of $k$. To obtain a contradiction, assume that we have 4 distinct values $i,j,p,q$ which are not fixed by $\pi$. Since $\pi$ is a bijection, the image of a non-fixed point is also not fixed, so we can assume that $\pi(i) = j$ and $\pi(p)=q$. 

Since $\phi$ and $\sigma$ are compatible and $\pi(i)=\pi(j)$, we see that $\sigma(\theta^i)$ is $c_1$-adjacent to $\phi(\theta^i)=\sigma(T_z(\theta^j))$. Since $\sigma$ is an embedding, this means that $\theta^i$ is $c_1$-adjacent to $T_z(\theta^j)$. This is possible only when $z=\theta^k$ for some $k\in \{i,j\}$.

Applying the same reasoning as the above paragraph to $\pi(p)=\pi(q)$ gives $z=\theta^k$ for $k\in \{p,q\}$, which is a contradiction since $\{i,j\}$ is disjoint from $\{p,q\}$. Thus we conclude that  $\pi(k)=k$ for all except at most 3 values of $k$. This means that when we express $\pi$ as a composition of transpositions, $\pi$ must either be identity, or $\pi=\tau_{i,j}$ for some $i\neq j$, or $\pi = \tau_{i,j}\circ \tau_{j,k}$ for some distinct $i,j,k$. We finish the proof by considering these three cases individually, showing that the third is impossible.

If $\pi$ is the identity, then $\phi$ has one of the forms numbered \ref{equal} or \ref{F} in the Lemma.

If $\pi=\tau_{ij}$, then $\pi(i)=j$ and arguments above show that $z = \theta^k$ for some $k\in \{i,j\}$. If $k=j$, then $\phi = T_j \circ r_{\tau_{ij}}$ which is form \ref{FC}, and if $k=i$ then $\phi = T_i\circ r_{\tau_{i,j}}$ which is also form \ref{FC}.


We will show that the remaining case, that $\pi = \tau_{ij}\circ \tau_{jk}$ for distinct $i,j,k$, is impossible. In this case we have $\pi(i)=j$ and arguments above show that $z=\theta^k$ for some $k\in \{i,j\}$. Also we have $\pi(j)=\pi(k)$, so the same arguments show that $k\in \{j,k\}$. Since $i,j,k$ are distinct, we must have $k=j$, and so $\phi$ has the form:
\[ \phi = \sigma \circ T_j \circ \tau_{ij} \circ \tau_{jk}. \]
We will show that the form above leads to a contradiction, rendering this case impossible.

We have:
\[ \phi(\theta^{ij}) =  \sigma \circ T_j \circ \tau_{ij} \circ \tau_{jk}(\theta^{ij}) = \sigma \circ T_j \circ \tau_{ij} (\theta^{ik}) = \sigma \circ T_j (\theta^{jk}) = \sigma(\theta^k) \]
Since $\phi$ and $\sigma$ are compatible, $\sigma(\theta^{ij})$ is $c_1$-adjacent to $\phi(\theta^{ij}) = \sigma(\theta^k)$. Since $\sigma$ is an embedding, $\sigma(\theta^{ij})$ being $c_1$-adjacent to $\sigma(\theta^k)$ implies that $\theta^{ij}$ is $c_1$-adjacent to $\theta^k$, which is a contradiction.
\end{proof}


Proving the desired properties of our main chain map will require a detailed discussion of noninjective cubes having at least one injective face. For such cubes, we will show that the injective faces must occur in particular arrangements.
%
%
%
We begin with two lemmas about a noninjective cube in which one face is injective but has different image from the opposite face.

\begin{lem}\label{biglemma}
Let $\sigma$ be a noninjective singular $q$-cube with $q\ge 2$, such that some $i$-face is injective, but the image of this face is not equal to the image of the opposite face.
Then there is some $j\neq i$ such that
$\sigma\circ T=\sigma\circ T\circ \tau_{ij}$, where $T \in \{\id,T_i\}$.
\end{lem}

\begin{proof}
Our strategy is to first show that $\sigma \circ T = \sigma \circ T \circ \tau_{ij}$, where $T=T_z$ for some $z\in \Z_2^q$, and then to show that in fact we must have $T=\id$ or $T=T_i$.

By replacing $\sigma$ with $\sigma \circ T_i$ 
if necessary, we assume without loss of generality 
that $A_i\sigma$ is injective, and $A_i\sigma(I^{q-1}) \neq B_i\sigma(I^{q-1})$. 
This means that there is some  $x\in I^{q}$ 
such that $\sigma(x)\in
A_i
\sigma(I^{q-1})$ but $\sigma(x)\notin 
B_i\sigma(I^{q-1})$. 
In fact, $\sigma(x)\notin\sigma(I^q-\{x\})$, 
because 
$A_i\sigma$ is injective.

Let $x=\theta^{i_1\cdots i_n}$
be the minimal element (lexicographically) 
with $\sigma(x)\notin\sigma(I^q-\{x\})$.
Note that $i\notin\{i_1,\ldots,i_n\}$.
Let 
$T_x = T_{i_1} \circ \dots \circ T_{i_n}$, and 
$\bar\sigma = \sigma \circ T_x$.
Since
$i\notin\{i_1,\ldots,i_n\}$, we have
$A_i\bar\sigma$ injective and $B_i\bar\sigma$
noninjective.
We also have $\bar\sigma(\0)=\sigma(x)$, which means that
$\bar\sigma(\0)\notin\bar\sigma(I^q-\{\0\})$.

We will show that there is some index $j\neq i$, 
with
$\bar\sigma=\bar\sigma\circ \tau_{ij}$ to get 
the desired result.


It will suffice to show that there is some $j$ with
\begin{equation}\label{eqn_Type1ABBA_elt}
\bar\sigma\left(\theta^{ik_1\cdots k_n}\right)
=\bar\sigma\left(\theta^{jk_1\cdots k_n}\right),
\end{equation}
for all indices $k_1,\dots,k_n$ different from $i$ and $j$. We prove \eqref{eqn_Type1ABBA_elt} by induction on $n$.

For $n=0$, we must show that $\bar\sigma\left(\theta^i\right)=\bar\sigma
\left(\theta^j\right)$ for some $j$. To obtain a contradiction, assume that $\bar \sigma(\theta^i) \neq \bar \sigma(\theta^j)$ for all $i,j$. 
Note that the element 
$\bar\sigma\left(\theta^i\right)$
is not equal to $\bar\sigma(\0)$ 
because $\bar\sigma(\0)\notin
\bar\sigma(I^q-\{\0\})$
and so from injectivity of $A_i\bar\sigma$, 
$\bar\sigma\vert_{S_0\cup S_1}$ is injective. 

Now for any $k \neq i$, the point $\bar\sigma(\theta^{ik})$ is a common neighbor 
of $\bar\sigma(\theta^i)$ and $\bar\sigma(\theta^k)$. 
This point $\bar\sigma(\theta^{ik})$ is distinct for different $k$, for if $\sigma(\theta^{ik_1})=\sigma(\theta^{ik_2})$, for $k_1\neq k_2$, then $\sigma(\theta^{ik_1})=\sigma(\theta^{ik_2})$ and $\sigma(\0)$ has three common neighbors, namely $\sigma(\theta^{k_1}),\sigma(\theta^{k_2})$ and $\sigma(\theta^{i})$, which is impossible from Corollary \ref{cor_comn_nbrs}.
And so $\bar\sigma\vert_{S_0\cup S_1\cup S_2}$ 
is injective.
But then,
by Lemma \ref{lem_sig_inj} 
$\bar\sigma$ is injective. This contradicts our assumption that $\sigma$ is noninjective, because $\sigma(I^q)=\bar\sigma(I^q)$.

Therefore there is an index $j$ with $\bar\sigma
\left(\theta^i\right)=\bar\sigma\left(\theta^j\right)
$, and the $n=0$ case of our inductive argument for \eqref{eqn_Type1ABBA_elt} is proved.

For $n=1$, we show that for any index $k$ different from $i$ and $j$,
$\bar\sigma(\theta^{ik})=\bar\sigma(\theta^{jk})$:
$\bar\sigma\left(\theta^{ik}\right)$ is a common neighbor of $\bar\sigma(\theta^i)=\bar\sigma(\theta^j)$ and $
\bar\sigma(\theta^k)$.
By injectivity of $A_i\bar\sigma$ and Lemma 
\ref{cube_unique},
the only common neighbors of $\bar\sigma(\theta^j)$ and $
\bar\sigma(\theta^k)$ are 
$\bar\sigma(\0)$ and $\bar\sigma(\theta^{jk})$, 
which means $
\bar\sigma\left(\theta^{ik}\right)=\bar\sigma
\left(\theta^{jk}\right)$, because 
$\bar\sigma(\0)\notin\bar\sigma(I^q-\{\0\})$.

For $n>1$, note that 
$\bar\sigma(\theta^{ik_1\cdots k_n})$
is a common neighbor of at least three elements:
$\bar\sigma(\theta^{k_1\cdots k_n})$,
and any elements of the form
$\bar\sigma(\theta^{ik_1\cdots \hat{k_l}\cdots k_n})$ for 
$l\in\{1,\ldots,n\}$.
By induction $\bar\sigma(\theta^{ik_1\cdots \hat{k_l}\cdots k_n})
=\bar\sigma(\theta^{jk_1\cdots \hat{k_l}\cdots k_n})$,
and by Corollary \ref{cor_grt_thn_3} and injectivity
of $A_i\sigma$, the only common neighbor of 
$\bar\sigma(\theta^{k_1\cdots k_n})$
and all $\bar\sigma(\theta^{jk_1\cdots 
\hat{k_l}\cdots k_n})$ for $l\in\{1,\ldots,n\}$
is
$\bar\sigma(\theta^{jk_1\cdots k_n})$.
Therefore, $\bar\sigma(\theta^{ik_1\cdots k_n})=
\bar\sigma(\theta^{jk_1\cdots k_n})$, establishing \eqref{eqn_Type1ABBA_elt}.

Thus far, we have shown that $\sigma \circ T = \sigma \circ T \circ \tau_{ij}$, where $T = T_x = T_{i_1} \circ \dots \circ T_{i_n}$ and $\tau_{ij}$ is the transposition of coordinates $i$ and $j$.

Since the various $T_k$ commute with one another, we can write $T = T' \circ G$, where $T'$ is a composition of axis reflections in dimensions $i$ and $j$, and $G$ is a composition of axis reflections in dimensions other than $i$ and $j$. This means that $T' \in \{\id, T_i,T_j,T_i\circ T_j\}$.

We have 
$\sigma \circ T' \circ G= \sigma \circ T' \circ 
G\circ \tau_{ij}.$
But $G$ commutes with $\tau_{ij}$ since $G$ consists of reflections in dimensions other than $i$ and $j$.
Thus we can write
\[\sigma \circ T' \circ G= \sigma \circ T' \circ \tau_{ij} \circ G.\]
Since $G$ is 
a bijection and a self inverse, 
we may cancel $G$ to obtain  
\[ \sigma \circ T' = \sigma \circ T' \circ \tau_{ij}, \]
where $T'\in\{\id, T_i, T_j, T_i\circ T_j\}$. If $T'\in\{\id, T_i, T_j\}$, then we are done. So it remains to consider the case when $T'=T_i\circ T_j$.


In this case we have $\sigma\circ T_i\circ T_j=\sigma\circ T_i\circ T_j \circ \tau_{ij} = \sigma \circ \tau_{ij} \circ T_i\circ T_j $, and composition with $T_i\circ T_j$ will give $\sigma = \sigma \circ \tau_{ij}$ as desired.
\end{proof}

\begin{lem}\label{oppfacesflips}
Let $\sigma$ be a singular $q$-cube with $q\ge 2$, such that $A_i\sigma=B_i\sigma \circ T_j$ for some $i$ and $j$. Then
$\sigma=\sigma\circ \tau_{ij}$.
\end{lem}
\begin{proof}
From Proposition \ref{faceformulas}, $A_i\sigma=B_i\sigma \circ T_j$ gives
 $B_jA_i\sigma=A_jB_i\sigma$. Now, for $t_i=t_j$, 
 we easily get 
 $\sigma(t_1,\ldots,t_q)=\sigma\circ \tau_{ij}
 (t_1,\ldots,t_q)$. Whereas  for $t_i\neq t_j$, we 
 get $\sigma(t_1,\ldots,t_q)=
 \sigma\circ \tau_{ij}(t_1,\ldots,t_q)$ from 
 $B_jA_i\sigma=A_jB_i\sigma$.
\end{proof}
\begin{thm}\label{newtypes}
Let $\sigma$ be a nondegenerate singular $q$-cube 
in $(X,c_1)$ with $q\geq 2$ and at least one face injective. Then one of the following cases holds for $\sigma$: 
\begin{enumerate}
\item For some $i$, $A_i\sigma$ and $B_i\sigma$ are injective, with each being the rotation of the other. \label{twofaces}
\item There are $i,j$ with $\sigma\circ T = \sigma\circ T \circ \tau_{ij}$, where $T\in \{\id, T_i\}$.\label{Ttau}
\end{enumerate}
\end{thm}
\begin{proof}
 If $\sigma$ has an injective face opposite a  face with a different image, then Lemma \ref{biglemma} applies, and Case \ref{Ttau} holds. So for the remainder of the proof, we consider only when all injective faces of $\sigma$ have image equal to their opposite faces.

Since opposite faces are always compatible, we may apply Lemma \ref{lem_compatible_embeddings}. 
Suppose for some $i$, $A_i\sigma$ and $B_i\sigma$ are injective.

In Case \ref{equal} of the lemma, we would have $A_i\sigma = B_i\sigma$, but this is impossible because $\sigma$ is nondegenerate.

In Case \ref{F} of that lemma, $A_i\sigma = B_i\sigma \circ T_j$, for some $j$ but then by Lemma \ref{oppfacesflips}, we have $\sigma = \sigma \circ \tau_{ij}$.

In Case \ref{FC} of Lemma \ref{lem_compatible_embeddings}, we immediately obtain Case 1 of the present theorem.

\end{proof}
\begin{lem}\label{rot_onlyinjfaces}
Let $\sigma$ be a nondegenerate singular $q$-cube such that for some $i$, $A_i\sigma$ and $B_i\sigma$ are injective and each is a rotation of the other. Then $A_i\sigma$ and $B_i\sigma$ are the only injective faces of $\sigma$.
\end{lem}
\begin{proof}
We prove the given statement only for the cases when $A_i\sigma=B_i\sigma\circ T_j\circ \tau_{jk}$ and $j<k<i$, while the proofs for the other cases of $i,j$ and $k$ are similar. 
Using the assumptions above we get: $A_jA_kA_i\sigma=B_jA_kB_i\sigma$ because
\begin{align*}
A_jA_kA_i\sigma(t_1,\ldots,t_q)
&=A_i\sigma(t_1,\ldots,t_j=0,\ldots,t_k=0,\ldots,t_q)\\
&=B_i\sigma\circ T_j\circ \tau_{jk}(t_1,\ldots,t_j=0,\ldots,t_k=0,\ldots,t_q)\\
&=B_i\sigma(t_1,\ldots,t_j=1,\ldots,t_k=0,\ldots,t_q)\\
&=B_jA_kB_i\sigma(t_1,\ldots,t_q)
\end{align*}
Similarly, we get the following equations:
\begin{align*}
B_jA_kA_i\sigma&=B_jB_kB_i\sigma\\
A_jB_kA_i\sigma&=A_jA_kB_i\sigma\\
B_jB_kA_i\sigma&=A_jB_kB_i\sigma
\end{align*}
Comparing both the sides of the 4 equations above, we see that $A_k\sigma$,$B_j\sigma$,$A_j\sigma$ and $B_k\sigma$ are noninjective, respectively.
The equations above easily give the noninjectivity of the other faces, \textit{e.g.} for $l<j$, $A_l\sigma$ is noninjective, because $A_lA_jA_kA_i\sigma=A_lB_jA_kB_i\sigma$.
\end{proof}
%
 
\section{The chain map $\beta$, and functoriality of $H_q^{c_1}(X)$}\label{sec_chainmap}
Now we consider our main goal for the paper: to define a chain map $\beta$ from the chain complex of $dH_q$ to the chain complex of $H^{c_1}_q$. This $\beta$ will be defined in terms of the orientation of an injective cube.

\begin{Def}
For a cube automorphism $\alpha:I^q \to I^q$, we define the \emph{orientation} $o^\alpha$ of $\alpha$ as follows: if $\alpha(t) = z \oplus \pi(t)$ for $z\in \Z_2^q$ and $\pi\in \Sigma_q$, then $o^\alpha = (-1)^{\#z} \cdot (-1)^\pi$, where $\#z$ is the number of nonzero coordinates of $z$, and $(-1)^\pi$ is the parity of the permutation $\pi$.

For an injective cube $\sigma: I^q \to X$ expressed as in Theorem \ref{singularhyperoctahedral} as $\sigma(t) = M(z\oplus \pi(t)) + x$, we define the orientation of $\sigma$ as the orientation of the cube automorphism part, that is $o^\sigma = (-1)^{\#z} \cdot (-1)^\pi$.
\end{Def}


Now we are ready to define our chain map $\beta$:
\begin{Def}
For each $q$, we define a homomorphism $\beta_q:dC_q(X) \to C_q^{c_1}(X)$ on generators as follows: Let $\sigma$ be a singular $q$-cube, and we define
\[ \beta_q(\sigma) = \begin{cases} 0 &\text{ if $\sigma$ is noninjective}, \\
o^\sigma \cdot \sigma(I^q) &\text{ if $\sigma$ is injective}. 
\end{cases}\]
\end{Def}

The remainder of the section will prove that $\beta$ is a chain map.
This will require that we show the two boundary operators for $dC_q$ and $C^{c_1}_q$ are compatible. As a step towards this goal, we first show that there is a correspondence between the appropriate face operators.

\begin{lem}\label{c1orientations}
Let $\sigma$ be an injective cube with $\sigma(t) = M(z\oplus \pi(t)) + x$. Let $1 \le k_1 < \dots < k_q \le n$ such that column $i$ of $M$ is the standard basis vector $e_{k_i} \in \Z^n$.

Then
\[
A_i \sigma(I^{q-1}) = \begin{cases}
A_{k_{\pi(i)}}^{c_1}(\sigma(I^q))$ if $ z_{\pi(i)} = 0, \\
B_{k_{\pi(i)}}^{c_1}(\sigma(I^q))$ if $ z_{\pi(i)} = 1,
\end{cases} 
\]
and
\[
B_i \sigma(I^{q-1}) = \begin{cases}
B_{k_{\pi(i)}}^{c_1}(\sigma(I^q))$ if $ z_{\pi(i)} = 0, \\
A_{k_{\pi(i)}}^{c_1}(\sigma(I^q))$ if $ z_{\pi(i)} = 1,
\end{cases} 
\]
\end{lem}
\begin{proof}
We will prove the statement for $A_i\sigma(I^{q-1})$. The statement for the $B_i$ face is similar.

Since $\sigma(t) = M(z\oplus \pi(t)) + x$, the set $\sigma(I^q)-x$ is the span of the columns of $M$. Thus we have
\[ \sigma(I^q)-x = \spanop \{e_{k_1}, \dots, e_{k_q}\}, \]
where ``span'' indicates the set of all linear combinations of these vectors using coefficients in $\{0,1\}$. Thus we will have
\[ 
\begin{split} 
A_{k_{\pi(i)}}^{c_1}(\sigma(I^q))-x &= \spanop\{e_{k_1},\dots,\hat e_{k_{\pi(i)}}, \dots, e_{k_q}\} \\
B_{k_{\pi(i)}}^{c_1}(\sigma(I^q))-x &= \spanop\{e_{k_1},\dots,\hat e_{k_{\pi(i)}}, \dots, e_{k_q}\} + e_{k_{\pi(i)}}
\end{split}
\]
where the hat denotes omission.

Let $A_i\sigma(t) = M'(z'\oplus \pi'(t)) + x'$ as in Theorem \ref{faceperm}. We have $x'=\sigma(0)-M'z'$, where $M'$ is the matrix $M$ with column $\pi(i)$ deleted, and $z'$ is the vector $z$ with entry $\pi'(i)$ deleted. When $z_{\pi(i)} = 0$, these deletions are irrelevant and we have $Mz = M'z'$. If $z_{\pi(i)} = 1$, then $Mz = M'z' + e_{k_{\pi(i)}}$. Thus we have:
\[ \begin{split}
x' &= \sigma(0) - M'z' = \begin{cases} \sigma(0) - Mz &$ if $ z_{\pi(i)} = 0, \\
\sigma(0) - Mz + e_{k_i} &$ if $ z_{\pi(i)} = 1. 
\end{cases} \\
&= \begin{cases} x &$ if $ z_{\pi(i)} = 0, \\
x + e_{k_i} &$ if $ z_{\pi(i)} = 1. 
\end{cases}
\end{split}
\]

Since $A_i\sigma(t) = M'(z'\oplus \pi'(t)) + x'$ we have:
\[ A_i\sigma(I^{q-1}) - x' = \spanop\{e_{k_1},\dots,\hat e_{k_{\pi(i)}}, \dots, e_{k_q}\}, \]
and thus the above gives:
\[ \begin{split}
A_i \sigma(I^{q-1}) - x &= \begin{cases}
\spanop\{e_{k_1},\dots,\hat e_{k_{\pi(i)}}, \dots, e_{k_q}\} &$ if $ z_{\pi(i)} = 0, \\
\spanop\{e_{k_1},\dots,\hat e_{k_{\pi(i)}}, \dots, e_{k_q}\} + e_{k_{\pi(i)}} &$ if $ z_{\pi(i)} = 1.
\end{cases} \\ 
&= \begin{cases}
A_{k_{\pi(i)}}^{c_1}(\sigma(I^q))-x$ if $ z_{\pi(i)} = 0, \\
B_{k_{\pi(i)}}^{c_1}(\sigma(I^q))-x$ if $ z_{\pi(i)} = 1.
\end{cases} 
\end{split}
\]
and so 
\[
A_i \sigma(I^{q-1}) = \begin{cases}
A_{k_{\pi(i)}}^{c_1}(\sigma(I^q))$ if $ z_{\pi(i)} = 0, \\
B_{k_{\pi(i)}}^{c_1}(\sigma(I^q))$ if $ z_{\pi(i)} = 1,
\end{cases} 
\]
as desired. 
\end{proof}

Orientations also behave predictably with respect to the face operators as follows:
\begin{thm}\label{faceorientation}
Let $\sigma$ be an injective cube with $\sigma(t) = M(z\oplus \pi(t)) + x$. Then 
\[ o^{A_i\sigma} = o^{B_i\sigma} = (-1)^{z_{\pi(i)} + i + \pi(i)} o^\sigma. \]
\end{thm}
\begin{proof}
By Lemma \ref{faceperm} we can see that $o^{A_i\sigma} = o^{B_i\sigma}$, since their automorphism parts $z'$ and $\pi'$ are the same. So it suffices to show that $o^{A_i\sigma} = (-1)^{z_{\pi(i)} + i + \pi(i)} o^\sigma$. As in Lemma \ref{faceperm}, let $A_i\sigma = M'(z'\oplus \pi'(t)) + x'$. 

Since $z'$ is obtained from $z$ by omission of coordinate $\pi(i)$, we have $(-1)^{\#z} = (-1)^{z_{\pi(i)}} (-1)^{z'}$. Using \eqref{delparity}, we have:
\[ \begin{split}
o^{A_i\sigma} &=  (-1)^{\#z'} (-1)^{\pi'} = (-1)^{z_{\pi(i)}} (-1)^{\#z} (-1)^{\pi'} \\
&= (-1)^{z_{\pi(i)}} (-1)^{\#z} (-1)^{i+\pi(i)}(-1)^\pi = (-1)^{z_{\pi(i)}}  (-1)^{i+\pi(i)} (-1)^\pi o^\sigma
\end{split}
\]
as desired.
\end{proof}

Orientations also behave predictably with respect to translations and transpositions, which lets us easily prove some useful properties of $\beta_q$:
\begin{thm}\label{betaflipswap}
For a singular $q$-cube $\sigma$ with $q\ge 1$, and any $i,j$, we have
\[ \beta_q(\sigma \circ T_j) = \beta_q(\sigma\circ \tau_{ij}) = -\beta_q(\sigma). \]
\end{thm}
\begin{proof}
If $\sigma$ is noninjective, then $\sigma\circ T_j$ and $\sigma\circ \tau_{ij}$ are also noninjective, and so we have $\beta_q(\sigma \circ T_j) = \beta_q(\sigma\circ \tau_{ij}) = -\beta_q(\sigma) =0$.

Now assume that $\sigma$ is injective. Since $(\sigma\circ T_j)(I^q) = (\sigma \circ \tau_{ij})(I^q) = \sigma(I^q)$, it suffices to show that $o^{\sigma\circ T_j} = o^{\sigma\circ \tau_{ij}} = -o^\sigma$. Let $\sigma(t) = M(z\oplus \pi) + x$ as in Theorem \ref{singularhyperoctahedral}.

We have:
\[ o^{\sigma \circ \tau_{ij}} = (-1)^{\#z} (-1)^{\pi\circ \tau_{ij}} = (-1)^{\#z} \cdot -(-1)^\pi = -o^\sigma. \]
For $\sigma \circ T_j$, we have 
\[ \sigma \circ T_j(t) = M(z\oplus (\pi \circ T_j))+x = M((z\oplus e_{\pi(j)}) \oplus \pi) + x, \] 
and so
\[ o^{\sigma \circ T_j} = (-1)^{\# z + 1} (-1)^\pi = -(-1)^{\#z} (-1)^\pi = -o^\sigma \]
as desired.
\end{proof}

From the above and the definition of $S_{i,j}$ we obtain immediately:
\begin{cor}\label{cor_RSS}
For any singular $q$-cube $\sigma$ and $i<j$, 
\[ \beta(\sigma\circ S_{i,j})=(-1)^{j-i}\beta\sigma. \]

\end{cor}

We aim to prove that $\beta$ is a chain map. As a preliminary step, we show that $\beta\circ \partial$ behaves nicely with respect to the maps $T_i$ and $\tau_{ij}$. 
\begin{lem}\label{betaboundaryflipswap}
Let $\sigma$ be a $q$-cube with $i\neq j$. Then
\begin{align*}
\beta(\partial (\sigma \circ T_i)) &= -\beta(\partial \sigma) \\
\beta(\partial (\sigma \circ \tau_{ij})) &= -\beta(\partial \sigma)
\end{align*}
\end{lem}
\begin{proof}
Both identities follow from Proposition \ref{boundaryflipswap} and Theorem \ref{betaflipswap}. We will prove the second, which is more complicated. Assuming $i<j$, by Proposition \ref{boundaryflipswap} we have:
\begin{align*}
\beta(\partial(\sigma \circ \tau_{ij})) &= \sum_{k<i}(-1)^{k}(\beta(A_k\sigma\circ 
\tau_{i-1,j-1})-\beta(B_k\sigma\circ \tau_{i-1,j-1}))
\\
&\qquad+(-1)^i(
\beta(A_j\sigma\circ S_{i,j-1})-\beta(B_j\sigma\circ S_{i,j-1}))
\\
&\qquad+\sum_{i<k<j}(\beta(A_k\sigma\circ \tau_{i,j-1}) -
\beta(B_k\sigma\circ \tau_{i,j-1}))
\\
&\qquad+(-1)^j(\beta(A_i\sigma\circ 
S_{j-1,i})-\beta(B_i\sigma\circ S_{j-1,i}))\\
&\qquad+\sum_{j<k}
(-1)^k(\beta(A_k\sigma\circ \tau_{i,j})
-\beta(B_k\sigma\circ \tau_{i,j})).
\end{align*}
By Theorem \ref{betaflipswap}, all terms written inside the three summations will be negated by $\beta$. For the other terms, we apply Corollary \ref{cor_RSS}, and the above becomes:
\begin{align*}
\beta(\partial(\sigma \circ \tau_{ij})) &= -\sum_{k\not\in\{i,j\}} (-1)^{k}(\beta(A_k\sigma)-\beta(B_k\sigma))
\\
&\qquad+(-1)^i (-1)^{j-1-i}(
\beta(A_j\sigma)-\beta(B_j\sigma)) \\
&\qquad+(-1)^j(-1)^{i-j+1}(\beta(A_i\sigma)-\beta(B_i\sigma))\\
&= -\sum_{k} (-1)^{k}(\beta(A_k\sigma)-\beta(B_k\sigma)) = -\beta(\pd(\sigma))
\end{align*}
as desired.
\end{proof}

Now we are ready for our main result of the section.
\begin{thm} 
$\beta:dC_q(X)\to C_q^{c_1}(X)$ is a chain map.
\end{thm}
\begin{proof}
For any singular $q$-cube $
\sigma$, we need to show that $\beta\partial\sigma=
\partial^{c_1}\beta\sigma$. 

If $\sigma$ is noninjective, and all faces of 
$\sigma$ are also noninjective, then we have $\beta\partial\sigma=0=
\partial^{c_1}\beta\sigma$.

Next we consider the case when $\sigma$ is injective. Let $\sigma(t) = M(z\oplus \pi(t)) + x$ be the expression given by Theorem \ref{singularhyperoctahedral}, and let $1\le k_1< \dots < k_q \le n$ be as in Lemma \ref{c1orientations}, so that column $i$ of $M$ is the standard basis vector $e_{k_i}$.

Then by Theorem \ref{faceorientation} we have:
\begin{align*}
\beta(\pd \sigma) &= \beta\left( \sum_{i=1}^q (-1)^i (A_i\sigma - B_i\sigma) \right) \\
&= \sum_{i=1}^q (-1)^i (o^{A_i\sigma} (A_i\sigma(I^{q-1})) - o^{B_i\sigma} (B_i\sigma(I^{q-1}))) \\
&= \sum_{i=1}^q (-1)^i (-1)^{z_{\pi(i)}+i + \pi(i)} o^\sigma (A_i\sigma(I^{q-1}) - B_i\sigma(I^{q-1})) \\
&= o^\sigma \sum_{i=1}^q (-1)^{z_{\pi(i)}+ \pi(i)}  (A_i\sigma(I^{q-1}) - B_i\sigma(I^{q-1})).
\end{align*}

By Lemma \ref{c1orientations} we have 
\[ (-1)^{z_{\pi(i)}} (A_i\sigma(I^{q-1}) - B_i\sigma(I^{q-1})) = A^{c_1}_{k_{\pi(i)}}\sigma(I^{q}) - B^{c_1}_{k_{\pi(i)}}\sigma(I^{q}). \]
Thus we continue the above to obtain:
\[
\begin{split}
\beta(\pd \sigma) &= o^\sigma \sum_{i=1}^q (-1)^{\pi(i)} (A^{c_1}_{k_{\pi(i)}}\sigma(I^{q}) - B^{c_1}_{k_{\pi(i)}}\sigma(I^{q})) \\
&= o^\sigma \sum_{j=1}^n  (-1)^{j} (A^{c_1}_{j}\sigma(I^{q}) - B^{c_1}_{j}\sigma(I^{q})) \\
&= o^\sigma \pd^{c_1}(\sigma(I^q)) = \pd^{c_1}(o^\sigma \sigma(I^q)) = \pd^{c_1}(\beta\sigma)
\end{split}
\]
as desired.

It remains to consider
the case when $\sigma$ is noninjective 
and at least one face of $\sigma$ is injective. 
In this case we automatically have $\beta\sigma =0$, and thus $\pd^{c_1}(\beta\sigma) =0$, and so we must show that $\beta(\pd \sigma) = 0$.

By Theorem \ref{newtypes}, either $A_i\sigma$ and $B_i\sigma$ are rotations of each other, or we have $\sigma \circ T = \sigma \circ T \circ \tau_{ij}$ for some $T \in \{\id, T_i\}$. We consider these cases separately.

First we consider the case where $A_i\sigma$ and $B_i\sigma$ are rotations of each other. Note that $\beta(A_i\sigma) = \beta(B_i\sigma)$ by Theorem \ref{betaflipswap} and $A_i\sigma$ and $B_i\sigma$ are the only injective faces by Lemma \ref{rot_onlyinjfaces} and thus we have 
\[ \beta(\pd \sigma) = (-1)^i(\beta(A_i\sigma) - \beta(B_i\sigma)) = 0 \]
as desired.

Now we consider the case where $\sigma \circ T = \sigma \circ T \circ \tau_{ij}$ for some $T \in \{\id, T_i\}$. By Lemma \ref{betaboundaryflipswap} we have $\beta(\partial(\sigma \circ T)) = -\beta(\partial(\sigma\circ T\circ \tau_{ij}))$. But since $\sigma\circ T= \sigma\circ T \circ \tau_{ij}$, this means that $\beta(\partial(\sigma\circ T)) = 0$. Thus, again using Lemma \ref{betaboundaryflipswap}, we have $\beta(\partial\sigma) = 0$.
\end{proof}

For the sake of our discussion we restate 
\cite[Conjecture~6.3]{Chris_Homotopy_relations}:
\begin{conj}\label{conjecture}
The chain map $\beta:C_q(X) \to C^{c_1}_q(X)$ induces an isomorphism $H_q(X) \cong H_q^{c_1}(X)$.
\end{conj}

In \cite{Chris_Homotopy_relations} it was not even known that $\beta$ was a chain map. Here we have shown that it is a chain map, and is surjective. We still are unable to determine if it induces an injection. 

\subsection{Functoriality and homotopy invariance of $H^{c_1}_q$}

We quote the definition of (digital) homotopy 
from \cite{Boxer_99}.
For integers, $a,b$ with $a<b$, the set $[a,b]_\mathbb{Z}
=\{a,a+1,\ldots,b\}$ is called a digital interval.
Let $f,g:X\to Y$ be continuous maps
between digital images $X$ and $Y$.
If there is an integer $k>0$ and a function
$F:X\times [0,k]_\mathbb{Z}\to Y$ satisfying:
\begin{itemize}
\item for all $x\in X$, we have $F(x,0)=f(x)$
and $F(x,k)=g(x)$;
\item for all $x\in X$, the induced function
$F_x:[0,k]_\mathbb{Z}\to Y$ defined by
$F_x(t)=F(x,t)$ is continuous;
\item for all $t\in [0,k]_\mathbb{Z}$, the induced function
$F_t:X\to Y$ defined by
$F_t(x)=F(x,t)$ is continuous.
\end{itemize}
Then $F$ is a \emph{homotopy} from $f$ to $g$, and $f$ and $g$
are \emph{homotopic}.

We say that the digital images $X$ and $Y$
are \emph{homotopy equivalent} or have the same \emph{homotopy 
type},
if there is a continuous function 
$f:X\to Y$ and a continuous function 
$g:Y\to X$, such that $g\circ f$ is homotopic to the identity
map on $X$ and $f\circ g$ is homotopic to the identity
map on $Y$.


For a continuous map  $f:X\to Y$  
between digital images $X$ and 
$Y$, the existence of a chain map
$f_\#^{c_1}:C_q^{c_1}(X)\to C_q^{c_1}(Y)$ was 
conjectured in 
\cite{Chris_Homotopy_relations}.
This conjecture was proven in low dimensions using computer enumerations, but we provide a complete analytical proof below.

In \cite{Jamil_DigCubSingHom} it was shown that $f_\#: C_q(X) \to C_q(Y)$ defined by $f_{\#q}(\sigma) = f\circ \sigma$ gives a chain map in the singular chain group. For the $c_1$-chain groups, we define $f^{c_1}_\#:C_q^{c_1}(X) \to C_q^{c_1}(Y)$ for some elementary cube $Q\subset X$ as follows: let $\sigma_Q:I^q \to Q$ be the affine embedding given by $J(t) = Mt + x$, where $x$ is the lexicographical minimum element of $Q$, and $M$ is the $q\times n$ matrix whose columns are the standard basis vectors $e_{k_i}$, where $1\le k_1<\dots<k_q\le n$ and $\{k_1,\dots,k_q\}$ is the set of coordinates in which the elementary intervals defining $Q$ are nondegenerate.
Then we define $f_\#^{c_1}(Q) = \beta_q(f \circ \sigma_Q)$. 

\begin{thm}\label{inducedchainmap}
The homomorphism $f^{c_1}_\#:C_q^{c_1}(X) \to C_q^{c_1}(Y)$ is a chain map.
\end{thm}
\begin{proof}
From the definition of $f^{c_1}_\#$, we have $\beta_q \circ f_{\#q} = f_{\#q}^{c_1}\circ \beta_q$ for any $q$. Also recall that in the singular chain groups, $f_\#$ defined above is a chain map, as is $\beta$. Let $Q$ be an elementary $q$-cube in $X$. Then we have:
\[ \begin{split}
\partial_q^{c_1}(f^{c_1}_{\#q}(Q)) &= \partial_q^{c_1}(\beta_q(f\circ \sigma_Q)) = \beta_{q-1}(\partial_q(f_{\#q}(\sigma_Q))) = \beta_{q-1}(f_{\#q-1}(\partial_q(\sigma_Q))) \\
&= f^{c_1}_{\#q-1}(\beta_{q-1}(\partial_q(\sigma_Q))) = f^{c_1}_{\#q-1}(\partial^{c_1}_q(\beta_q(\sigma_Q))) = f^{c_1}_{\#q-1}(\partial^{c_1}_q(Q))
\end{split}
\]
as desired.
\end{proof}


By the above, we immediately obtain:

\begin{thm}
\label{cor_functoriality_c_1_homology}
$H^{c_1}_q$ is a functor on the category of $c_1$-digital images and $(c_1,c_1)$-continuous functions.
\end{thm}

In \cite{Chris_Homotopy_relations}, Theorem \ref{inducedchainmap} was conjectured but not proven. Anticipating an eventual proof, Theorem 7.2 of \cite{Chris_Homotopy_relations} shows that if $f^{c_1}_\#$ is indeed a chain map, then homotopic maps induce the same homomorphisms in homology. 
Now that we have proven $f^{c_1}_\#:C_q^{c_1}(X) \to C_q^{c_1}(Y)$ to be a chain map for all $q$, the homotopy invariance of $H^{c_1}_q$ follows from Theorem 7.2 of \cite{Chris_Homotopy_relations}.
\begin{thm}
\label{cor_homotopy_inv_Hc1}
For digital image $X$ and $Y$ and continuous function $f,g:X\to Y$, if $f$ is homotopic to $g$, then the induced homomorphisms $f^{c_1}_*,g^{c_1}_*:H_q^{c_1}(X)\to H_q^{c_1}(Y)$ are equal. 

In particular, for homotopy equivalent digital images $X$ and $Y$,  we have $H^{c_1}_q(X)\cong H^{c_1}_q(Y)$.
\end{thm}

\section{Noninjective cubes in discrete singular cubical homology of graphs}\label{sec_final}
The conjecture that $\beta$ induces an isomorphism would imply that noninjective cubes do not need to be considered when computing $dH_q(X)$ for a $c_1$-digital image $X$. It is natural to ask if the same would be true in general for the discrete singular cubical homology of any graph. 

From the previous section, $
\beta$ is a chain map, and it is clear that 
$\beta$ is surjective.
Denote the quotient chain complex 
$dC_q(X)/\ker\beta$ as
$\overline C_q(X)$, and its $q^{th}$-homology
 as $\overline H_q(X)$. 
Then $\beta$ induces an isomorphism 
$\overline H_q(X) \cong H_q^{c_1}(X)$ for all 
$q$.

We can view the homology groups $\bar H_q(X)$ as a sort of ``injective singular cubical homology'' for the digital image $X$, and these are easily computable because they are isomorphic to $H_q^{c_1}(X)$. 

In fact an injective singular cubical homology theory can be defined without using $\beta$ as follows: When $G$ is any graph (not necessarily a grid graph), we may still define what it means for two injective cubes $\sigma,\psi: I^q \to G$ to have the same or opposite orientations. (We may define $\sigma$ and $\psi$ to have the same orientation if and only if $\sigma \circ \psi^{-1}:I^q \to I^q$ has orientation $+1$.)

Now let $C_q(G)$ be the singular cubical chain complex of $G$, and let $\bar C_q(G)$ be the quotient obtained by identifying two injective cubes $\sigma$ and $\psi$ when they have the same image and same orientation, identifying $\sigma$ with $-\psi$ if they have the same image and different orientations, and identifying any noninjective cube to 0.

Computation of this homology group $\bar H_q(G)$ is far easier than for $dH_q(G)$, since it requires enumerating only the injective singular $q$-cubes, rather than all possible singular $q$-cubes. Conjecture \ref{conjecture} would imply that $dH_q(G) \cong \bar H_q(G)$ when $G$ is a grid graph. 

But $dH_q(G)$ is not in general isomorphic to $\bar H_q(G)$, as the following counterexample demonstrates. The authors would like to thank Georg Wille for bringing this example to our attention.
\begin{ex}
Let $G$ be the following graph: 
\[ 
\begin{tikzpicture}[vertex/.style = {draw, circle}] 
\node[vertex] (a) at (1,1) {$a$};
\node[vertex] (b) at (0,0) {$b$};
\node[vertex] (c) at (1,0) {$c$};
\node[vertex] (d) at (2,0) {$d$};
\node[vertex] (e) at (1,-1) {$e$};
\draw (a) -- (b) -- (e) -- (d) -- (a);
\draw (b) -- (c) -- (d);
\end{tikzpicture}
\]
It is routine to check that the loops given by $(abcd)$, $(bedc)$ and $(abed)$ give three injective $2$-cubes whose sum is a cycle in $\bar C_2(G)$. But this cycle is not a boundary of any chain in $\bar C_3(G)$, because there are no injective $3$-cubes in $G$. Thus $\bar H_2(G)$ is nontrivial. 

But $dH_2(G)$ is trivial, because $G$ is contractible. (The identity map on $G$ is homotopic to the map given by $(a,b,c,d,e) \mapsto (b,b,c,c,b)$, which is homotopic to a constant.) Thus $dH_2(G)$ is not isomorphic to $\bar H_2(G)$.
\end{ex}

The example above indicates that if indeed the induced  map $\beta:dH_q(X) \to H_q^{c_1}(X)$ is an isomorphism, then this must be due to the specific grid-like structure of $X$, since it will not be true for general graphs. 

We also see in the example above that the ``injective cubical homology groups'' $\bar H_q(X)$ are not preserved by homotopy equivalence. This indicates that, for general graphs, the homology theory $\bar H$ is not functorial. However, we have shown in Theorem \ref{cor_functoriality_c_1_homology} that in fact it is functorial for grid graphs. Apparently this is a special property of grid graphs, and it would be interesting to investigate if it holds for other categories of graphs.

\noindent\textbf{Acknowledgements}\\
This paper is based on work conducted while the first author was a graduate student at the Institute of Business Administration (IBA), Karachi, Pakistan. This work is also partially supported by the National Science Foundation under Grant No. DMS-1928930 and by the Alfred P. Sloan Foundation under Grant No. G-2021-16778, during the first author's residency at the Simons Laufer Mathematical Sciences Institute (formerly MSRI) in Berkeley, California.

The authors would like the thank Curtis Greene, Georg Wille, and Volkmar Welker for helpful comments on an earlier version.

\end{document}